\documentclass[12pt,twoside,english,reqno,a4paper]{amsart}

\usepackage{listings,graphicx,amsmath,varioref,amscd,amssymb,color,bm,stmaryrd,amsthm,amsfonts,graphics,lineno}
\usepackage[english]{babel}

\usepackage{hyperref}

\newlength{\defbaselineskip}
\theoremstyle{plain}
\newtheorem{defin}{Definition}[section]
\newtheorem{theorem}[defin]{Theorem}
\newtheorem{prop}[defin]{Proposition}
\newtheorem{lemma}[defin]{Lemma}
\newtheorem{conj}[defin]{Conjecture}

\newtheorem{remark}[defin]{Remark}
\numberwithin{equation}{section}

\newcommand{\dis}{\displaystyle}
\newcommand{\eps}{\varepsilon}
\newcommand{\NN}{\mathbb{N}}
\def\vp{\varphi}
\def\RN{\mathbb{R}^{N}}
\def\D{\nabla}
\def\div{\text{\text{div}}}
\def\into{\int_{\Omega}}
\def\nk{n_{k}}

\newcommand{\RR}{\mathbb{R}}

\def\sob{W^{1,p}_{0}(\Omega)}
\def\linf{L^{\infty}(\Omega)}
\def\luno{L^{1}(\Omega)}
\def\lp'n{(L^{p'}(\Omega))^{N}}
\def\lr{L^{r}(\Omega)}
\def\lru{L^{r+1}(\Omega)}
\def\lr1d{L^{(r+1)'}(\Omega)}
\def\dualsob{W^{-1,p'}(\Omega)}

\author[R. Durastanti]{Riccardo Durastanti 
\\Dipartimento di Scienze di Base e Applicate per l' Ingegneria, ``Sapienza" Universit\`a di Roma, Via Scarpa 16, 00161 Roma, Italy
\\riccardo.durastanti@sbai.uniroma1.it}


\begin{document}

\title{Regularizing effect for some $p$-Laplacian systems}

\maketitle

\section*{\textbf{Abstract}} 

We study existence and regularity of weak solutions for the following $p$-Laplacian system
$$
\begin{cases}
       -\Delta_p u+A\varphi^{\theta+1}|u|^{r-2}u=f, \quad &u\in\sob, \\
       -\Delta_p \vp=|u|^r\vp^\theta, \quad &\varphi\in\sob, 
\end{cases}
$$
where $\Omega$ is an open bounded subset of $\RN$ $(N\geq 2)$, $\Delta_p v:=\div(|\D v|^{p-2}\D v)$ is the $p$-Laplacian operator, for $1<p<N$, $A>0$, $r>1$, $0\leq\theta<p-1$ and $f$ belongs to a suitable Lebesgue space. In particular, we show how the coupling between the equations in the system gives rise to a regularizing effect producing the existence of finite energy solutions.

%\tableofcontents

\section{\textbf{Introduction}} 
\label{Sec1}

This paper has been motivated by the work of Benci and Fortunato \cite{bf}. In that work the authors, investigating the eigenvalue problem for the Schr\"odinger operator coupled with the electromagnetic field, studied the existence for the following system of Schr\"odinger-Maxwell equations in $\RR^3$
\begin{equation}
\label{pbbf}
\begin{cases}
-\frac{1}{2}\Delta u+\varphi u=\omega u, \\
-\Delta \varphi=4\pi u^2.
\end{cases}
\end{equation}
The existence of a solution of \eqref{pbbf} is proved by using a variational approach: the equations of the system are the Euler-Lagrange equations of a suitable functional that is neither bounded from below nor from above but has a critical point of saddle type. \\

Starting from this work, first Boccardo in \cite{b} then Boccardo and Orsina in \cite{bo} studied the related Dirichlet problem with a source term $f$
\begin{equation}
\label{pbbo}
\begin{cases}
-\Delta u+A\varphi |u|^{r-2}u=f, \quad &u\in W_0^{1,2}(\Omega), \\
-\Delta\varphi=|u|^r, \quad &\varphi\in W^{1,2}_0(\Omega),
\end{cases}
\end{equation}
where $\Omega$ is an open bounded subset of $\RN$ with $N>2$, $A>0$ and $r>1$. \\

In \cite{b} the existence of a weak solution $(u,\varphi)$ in $W^{1,2}_0(\Omega)\times W^{1,2}_0(\Omega)$ is proved if $f$ belongs to $L^{m}(\Omega)$, with $\dis m\geq\frac{2N}{N+2}=(2^*)'$, where $2^*$ is the Sobolev exponent, using once again that $(u,\varphi)$ is a critical point of a suitable functional. The author proves that if $\dis (2^*)'\leq m < \frac{2Nr}{N+2+4r}$, with $r>2^*-1$, the second equation of \eqref{pbbo} admits finite energy solutions even if the datum $|u|^r$ does not belong to the dual space $\dis L^{\frac{2N}{N+2}}(\Omega)$. \\
In \cite{bo} the authors improve this result by proving a regularizing effect also on the solution $u$ of the first equation of \eqref{pbbo}. Existence of a solution $(u,\varphi)$ in $W^{1,2}_0(\Omega)\times W^{1,2}_0(\Omega)$ is proved if $r>2^*$ and $f$ belongs to $L^m(\Omega)$, with $m\geq r'$. Then, in the case $\dis r'\leq m<(2^*)'$, the authors find a finite energy solution $u$ of the first equation of \eqref{pbbo} with data $f$ possibly not belonging to the dual space.  \\

In this paper we are concerned with the existence of solutions for the following nonlinear elliptic system that generalizes \eqref{pbbo}
\begin{equation}
\label{pbp}
\begin{cases}
       -\div(|\D u|^{p-2}\D u)+A\varphi^{\theta+1}|u|^{r-2}u=f, \quad &u\in\sob, \\
       -\div(|\D\varphi|^{p-2}\D\varphi)=|u|^r\vp^\theta, \quad &\varphi\in\sob, 
\end{cases}
\end{equation}
where $\Omega$ is an open bounded subset of $\RN$ with $N\geq 2$, $1<p<N$, $A>0$, $r>1$ and $0\leq\theta<p-1$. \\
In the case $\theta=0$ the system \eqref{pbp} becomes
\begin{equation}
\label{pbp1}
\begin{cases}
       -\div(|\D u|^{p-2}\D u)+A\varphi|u|^{r-2}u=f, \quad &u\in\sob, \\
       -\div(|\D\varphi|^{p-2}\D\varphi)=|u|^r, \quad &\varphi\in\sob.
\end{cases}
\end{equation}
For such value of $\theta$, we show how the regularizing effect proved in \cite{bo} can be improved, proving the existence of a weak solution $u$ in $\sob$ of the first equation of \eqref{pbp1} with $f$ belonging to $L^m(\Omega)$, with $(r+1)'\leq m<(p^*)'$. \\
Conversely, in the case $p=2$ and $0<\theta<1$ the second equation of the system \eqref{pbp} is sublinear. This fact does not allow us to use the same method as the previous case and we are not able to prove the regularizing effect on $u$. However, we generalize the results proved in \cite{b} (in which we recall that $p=2$ and $\theta=0$). \\

Without the aim to be complete, we refer to various developments of the paper \cite{bf} in which the equations are defined in $\mathbb{R}^3$ and the right hand side of the first equation of \eqref{pbbf} is replaced with a nonlinear function $g(x,u)$ with polynomial growth in $u$ (see e.g. \cite{a}, \cite{cv}, \cite{c}, \cite{dm}, \cite{hy}, \cite{kr}, \cite{r}). \\
As concerns semilinear elliptic systems we refer to \cite{df}, where the author proves existence, multiplicity and symmetry of solutions. In the case of elliptic systems with singular lower order terms see \cite{bo2}, \cite{dos}. \\

The paper is organized as follows. In Section \ref{Sec2} we deal with a regular datum for the first equation in \eqref{pbp}. We define the following functional 
$$J(z,\eta)=\frac{1}{p}\into |\D z|^p-\frac{A(\theta+1)}{pr}\into |\D\eta|^p+\frac{A}{r}\into (\eta^+)^{\theta+1}|z|^r-\into fz,$$
and we prove existence of a saddle point $(u,\varphi)$ of $J$ in $\sob\times\sob$ which is a weak solution of \eqref{pbp}. \\
In Section \ref{Sec3} we provide the approximation scheme that gives us estimates in the case $\theta=0$ and, by these estimates, we prove that there exists a solution in $\sob\times\sob$ of the system \eqref{pbp1} with $f$ possibly not belonging to the dual space. We give also a summability result on the solution $u$ of the first equation. \\
Section \ref{Sec4} is devoted to the case $0<\theta<p-1$. Once again by an approximation scheme we prove estimates that allow us to pass to the limit in the approximate equations and to prove the existence of a weak solution of \eqref{pbp}, with the datum $f$ in the dual space.

\section{\textbf{Regular data}} 
\label{Sec2}

Let us firstly prove the existence of a weak solution $(u,\varphi)$ of \eqref{pbp} with data $f$ in $L^m(\Omega)$, $m>\frac{N}{p}$. This solution is a saddle point of a functional defined on $\sob\times\sob$.
\begin{prop}
\label{p1}
Let $f$ in $L^{m}(\Omega)$, with $m>\frac{N}{p}$, and let $A>0$, $r>1$ and $0\leq\theta<p-1$. Then there exists a weak solution $(u,\varphi)$ of \eqref{pbp}. Moreover, $u$ and $\varphi$ are in $L^\infty(\Omega)$, $\varphi\geq 0$ and $(u,\varphi)$ is a saddle point of the functional defined on $W^{1,p}_{0}(\Omega) \times W^{1,p}_{0}(\Omega)$ as
\begin{equation}
\label{fun}
J(z,\eta)=
\begin{cases}
\frac{1}{p}\into |\D z|^p-\frac{A(\theta+1)}{pr}\into |\D\eta|^p+\frac{A}{r}\into (\eta^+)^{\theta+1}|z|^r-\into fz & \mathrm{if } \into (\eta^+)^{\theta+1}|z|^r<+\infty, \\
+\infty & \mathrm{otherwise}.
\end{cases}
\end{equation}
\end{prop}

\begin{proof}
Fix $\psi\in\sob$ and let $I_1$ be the functional defined on $\sob$ as $I_1(z):=J(z,\psi)$. We have, by H\"older's inequality and denoting by $C_s$ the constant of the Sobolev embedding theorem, that  
\begin{equation*}
I_1(z)\geq \frac{1}{p}\|z\|^p_{\sob}-\frac{A(\theta+1)}{pr}\|\psi\|^p_{\sob}-C_s\|f\|_{L^{(p^*)'}(\Omega)}\|z\|_{\sob}.
\end{equation*}
This implies that $I_{1}$ is coercive. Now we prove that $I_1$ is weakly lower semicontinuous, which is that if $z_n\rightharpoonup z$ in $\sob$ then
\begin{equation}
\label{sem1}
I_1(z)\leq \liminf_{n\to\infty}I_1(z_n).
\end{equation}
Since $f\in L^m(\Omega)\subset L^{(p^*)'}(\Omega)$ we have that that
\begin{equation*}
\lim_{n\to\infty}\into fz_n=\into fz.
\end{equation*}
As a consequence of Fatou's lemma, it also yields
\begin{equation*}
\frac{A}{r}\into (\psi^+)^{\theta+1}|z|^r\leq \liminf_{n\to\infty}\frac{A}{r}\into (\psi^+)^{\theta+1}|z_n|^r.
\end{equation*}
Then, by the weakly lower semicontinuity of the norm, we deduce \eqref{sem1}. Hence there exists a minimum $v$ of $I_1$ on $\sob$. Moreover, by the classical theory of elliptic equations, $v$ is the unique weak solution of the Euler-Lagrange equation 
\begin{equation}
\label{eq1}
-\div(|\D v|^{p-2}\D v)+A(\psi^+)^{\theta+1}|v|^{r-2}v=f, \quad v\in\sob.
\end{equation}
We have, thanks to the results in \cite{s}, that
\begin{equation}
\label{st1}
\|v\|_{\sob}+\|v\|_{L^{\infty}(\Omega)}\leq C_1\|f\|_{L^{m}(\Omega)}^{\frac{1}{p-1}},
\end{equation}
where $C_1$ is a positive constant not depending on $f$. We define $S:\sob\to\sob$ as the operator such that $v=S(\psi)$. Now we consider the functional defined on $\sob$ as $I_2(\eta):=J(v,\eta)$. As before, since $\theta<p-1$, we have that $-I_2$ is coercive and weakly lower semicontinuous. Then there exists a minimum $\zeta$ of $-I_2$, that is a maximum of $I_2$ on $\sob$. Let $I_3$ be a functional defined on $\sob$ as
$$I_3(\eta):=\frac{\theta+1}{p}\into |\D\eta|^p-\into (\eta^+)^{\theta+1}|v|^r.$$
Since $\zeta$ is a maximum of $I_2$, we have
\begin{align*}
\frac{A}{r}I_3(\zeta)&=-I_2(\zeta)+\frac{1}{p}\into |\D v|^p-\into fv \\
&\leq -I_2(\eta)+\frac{1}{p}\into |\D v|^p-\into fv=\frac{A}{r}I_3(\eta), \quad \forall \eta\in\sob,
\end{align*}
so that $\zeta$ is a minimum of $I_3$. We observe that $\zeta\geq0$ and $\zeta\not\equiv 0$ in $\Omega$. In fact we have 
\begin{align*}
I_3(\zeta)&=\frac{\theta+1}{p}\into |\D\zeta|^p-\into (\zeta^+)^{\theta+1}|v|^r\leq \frac{\theta+1}{p}\into |\D\zeta^+|^p-\into (\zeta^+)^{\theta+1}|v|^r=I_3(\zeta^+),
\end{align*}
then $\|\zeta\|_{\sob}\leq \|\zeta^{+}\|_{\sob}$ and so $\zeta^-$ is zero almost everywhere in $\Omega$. Now we show that $\zeta\not\equiv 0$. We consider $\lambda_1$ to be the first eigenvalue of $-\Delta_p$ while $\vp_1$ in $\sob$ is the associated eigenfunction, that is
\begin{equation*}
\begin{cases}
-\div(|\D \vp_1|^{p-2}\D \vp_1)=\lambda_1 |\vp_1|^{p-2}\vp_1 \quad &\text{in } \Omega, \\
\vp_1>0 \quad &\text{in } \Omega, \\
\vp_1=0 \quad &\text{on } \partial\Omega.
\end{cases}
\end{equation*}
Let $t>0$; computing $I_3$ in $t\vp_1$, we obtain
\begin{align*}
I_3(t\vp_1)&=\frac{(\theta+1)t^p}{p}\into |\D\vp_1|^p-t^{\theta+1}\into\vp_1^{\theta+1}|v|^r \\
&=\frac{(\theta+1)\lambda_1t^p}{p}\into \vp_1^p-t^{\theta+1}\into \vp_1^{\theta+1}|v|^r=c_1t^p-c_2t^{\theta+1},
\end{align*}
where $\dis c_1:=\frac{(\theta+1)\lambda_1}{p}\into \vp_1^p \in(0,+\infty)$ and $\dis c_2:=\into \vp_1^{\theta+1}|v|^r\in(0,+\infty]$. By taking $t$ such that $\dis c_1 t^{p-\theta-1}-c_2<0$, that is $\dis t<\left(\frac{c_2}{c_1}\right)^{\frac{1}{p-\theta-1}}$, we have $I_3(t\vp_1)<0$. Then $I_3(\zeta)<0=I_3(0)$ and $\zeta\not\equiv 0$. Since $\zeta$ is a nonnegative minimum of $I_3$, thanks to the results in \cite{bros}, it is the unique weak solution of the Euler-Lagrange equation 
\begin{equation}
\label{eq2}
-\div(|\D\zeta|^{p-2}\D\zeta)=|v|^r\zeta^\theta, \quad \zeta\in\sob.
\end{equation}
Following \cite{bo1}, we have that
\begin{equation}
\label{st2}
\|\zeta\|_{\sob}+\|\zeta\|_{\linf}\leq C_2\|v\|_{\linf}^{\frac{r}{p-\theta-1}},
\end{equation}
and we deduce, using \eqref{st1}, that
\begin{equation}
\label{st3}
\|\zeta\|_ {\sob}+\|\zeta\|_{\linf}\leq C\|f\|_{L^{m}(\Omega)}^{\frac{r}{(p-1)(p-\theta-1)}}=:R,
\end{equation}
where $C$ and $C_2$ are positive constants not depending on $f$ and $v$. Now we define $T:\sob\to\sob$ as the operator such that $\zeta=T(v)=T(S(\psi))$. We want to prove that $T\circ S$ has a fixed point by Schauder's fixed point theorem. By \eqref{st3} we have that $\overline{B_R(0)}\subset \sob$ is invariant for $T\circ S$. Let $\{\psi_n\}\subset\sob$ be a sequence weakly convergent to some $\psi$ and let $v_n=S(\psi_n)$. As a consequence of \eqref{st1}, there exists a subsequence indexed by $v_{n_k}$ such that
\begin{equation}
\label{con1}
\begin{array}{l}
v_{\nk} \rightarrow v \text{ weakly in } \sob, \text{ and a.e. in } \Omega, \\
v_{\nk} \rightarrow v \text{ weakly-* in } \linf. 
\end{array}
\end{equation}
Moreover, we have 
$$-\div(|\D v_{\nk}|^{p-2}\D v_{\nk})=f-A(\psi_{\nk}^+)^{\theta+1}|v_{\nk}|^{r-2}v_{\nk}=:g_{\nk},$$
and, using H\"{o}lder's inequality, the Poincar\'e inequality and \eqref{st1}, we obtain
$$\|g_{\nk}\|_{\luno}\leq \|f\|_{\luno}+A\|v_{\nk}\|_{\linf}^{r-1}\|\psi_{\nk}\|_{L^{\theta+1}(\Omega)}^{\theta+1} \leq \|f\|_{\luno} +AC_1\|f\|_{L^m(\Omega)}^{\frac{r-1}{p-1}}\|\psi_n\|_{\sob}^{\theta+1}\leq C.$$
Then, by Theorem 2.1 in \cite{bm}, we obtain that $\D v_{\nk}$ converges to $\D v$ almost everywhere in $\Omega$. Since 
$$\||\D v_{\nk}|^{p-2}\D v_{\nk}\|_{\lp'n}=\|v_{\nk}\|_{\sob}^{p-1}\leq C_1\|f\|_{L^{m}(\Omega)},$$
we deduce that 
\begin{equation}
\label{con2}
|\D v_{\nk}|^{p-2}\D v_{\nk} \rightarrow |\D v|^{p-2}\D v \text{ weakly in } \lp'n.
\end{equation}
We recall that $v_{\nk}$ satisfies
\begin{equation*}
\into |\D v_{\nk}|^{p-2}\D v_{\nk} \cdot\D w+A\into (\psi_{\nk}^+)^{\theta+1}|v_{\nk}|^{r-2}v_{\nk} w=\into fw, \quad \forall w\in\sob.
\end{equation*}
Letting $k$ tend to infinity, by \eqref{con1}, \eqref{con2} and Vitali's theorem, we have that
\begin{equation*}
\into |\D v|^{p-2}\D v \cdot\D w+A\into (\psi^+)^{\theta+1}|v|^{r-2}v w=\into fw, \quad \forall w\in\sob,
\end{equation*}
so that $v$ is the unique weak solution of \eqref{eq1} and it does not depend on the subsequence. Hence $v_n=S(\psi_n)$ converges to $v=S(\psi)$ weakly in $\sob$ and weakly-* in $\linf$. Then 
\begin{equation}
\label{con3}
|v_n|^r \rightarrow |v|^r \text{ strongly in } L^q(\Omega)\text{ } \forall q<+\infty \text{ and } \||v_n|^r\zeta_n^\theta\|_{\luno}\leq C.
\end{equation}
Using \eqref{st3}, \eqref{con3} and proceeding in the same way, we obtain that 
\begin{align}
\label{con4}
\begin{array}{l}
\zeta_n=T(v_n) \rightarrow \zeta=T(v) \text{ weakly in } \sob, \text{ and weakly-* in }\linf,\\
|\D\zeta_n|^{p-2}\D\zeta_n \rightarrow |\D\zeta|^{p-2}\D\zeta \text{ weakly in }\lp'n, 
\end{array}
\end{align}
and $\zeta$ is the unique weak solution of \eqref{eq2}. Now we want to prove that $\zeta_n$ converges to $\zeta$ strongly in $\sob$. In order to obtain this, by Lemma 5 in \cite{bmp}, it is sufficient to prove the following
\begin{equation}
\label{lim}
\lim_{n\to\infty}\into \left(|\D \zeta_n|^{p-2}\D\zeta_n-|\D\zeta|^{p-2}\D\zeta\right)\cdot\D\left(\zeta_n-\zeta\right)=0.
\end{equation}
We have that
\begin{align}
\label{lim2}
\into \left(|\D \zeta_n|^{p-2}\D\zeta_n-|\D\zeta|^{p-2}\D\zeta\right)\cdot\D\left(\zeta_n-\zeta\right)&=\into |\D\zeta_n|^p-\into |\D\zeta|^{p-2}\D\zeta\cdot\D\zeta_n \\
&-\into |\D\zeta_n|^{p-2}\D\zeta_n\cdot\D\zeta+\|\zeta\|_{\sob}^p \nonumber.
\end{align}
The second and the third term on the right hand side of \eqref{lim2} converge, by \eqref{con4}, to $\|\zeta\|_{\sob}^p$. Then it is sufficient to prove that
\begin{equation}
\label{lim3}
\lim_{n\to\infty} \|\zeta_n\|_{\sob}^p=\|\zeta\|_{\sob}^p.
\end{equation}
Since $\zeta_n$ is equal to $T(v_n)\geq0$, we have that
\begin{equation*}
\into |\D\zeta_n|^p=\into |v_n|^r\zeta_n^{\theta+1}.
\end{equation*}
By \eqref{con3} and Vitali's theorem, we deduce that
$$\lim_{n\to\infty}\into |v_n|^r\zeta_n^{\theta+1}=\into |v|^r\zeta^{\theta+1}=\|\zeta\|_{\sob}^p,$$
so that \eqref{lim3} is true and \eqref{lim} is proved. Hence we have proved that if $\psi_n$ converges to $\psi$ weakly in $\sob$ then $\zeta_n=T(S(\psi_n))$ converges to $\zeta=T(S(\psi))$ strongly in $\sob$. As a consequence we have that $T\circ S$ is a continuous operator and that $T(S(\overline{B_R(0)}))\subset\sob$ is a compact subset. Then there exists, by Schauder's fixed point theorem, a function $\varphi$ in $\sob$ such that $\varphi=T(S(\varphi))$ and, since $T(v)\geq 0$ for every $v$ in $\sob$, $\vp$ is nonnegative. Moreover let $u=S(\varphi)$, we have that $u$ is a minimum for $I_1$ and $\varphi$ is a maximum for $I_2$. Hence $(u,\varphi)$ is a saddle point of $J$ defined by \eqref{fun} and a weak solution of \eqref{pbp}.
\end{proof}

\section{\textbf{Existence and regularizing effect in the case $\theta=0$}} 
\label{Sec3}

In this section we assume $\theta=0$ and we study the regularizing effect on the existence of finite energy solutions of both equations even if the data do not belong to the dual space. We recall that the assumption on $\theta$ implies that we deal with the system \eqref{pbp1}. \\
We consider the datum $f$ in $\lr1d$ and a sequence $\{f_n\}$ such that 
\begin{equation*}
f_n\in\linf \text{, } |f_n|\leq |f| \text{ }\forall n\in\NN \text{ and } f_n\rightarrow f \text{ strongly in } \lr1d.
\end{equation*}
By Proposition \ref{p1}, there exists $(u_n,\vp_n)$ in $\sob \times \sob$ that satisfies
\begin{equation}
\label{papp1}
\begin{cases}
-\div(|\D u_n|^{p-2}\D u_n)+A\vp_n|u_n|^{r-2}u_n=f_n, \quad &(i),\\
-\div(|\D\varphi_n|^{p-2}\D\varphi_n)=|u_n|^r, \quad &(ii),
\end{cases}
\end{equation}
with $\vp_n\geq 0$, $u_n$ and $\vp_n$ in $L^\infty(\Omega)$. Choosing $u_n$ as test function in $(i)$ and $\vp_n$ in $(ii)$ of \eqref{papp1} we have
\begin{equation*}
\into |\D u_n|^p+A\into \vp_n |u_n|^r=\into f_nu_n, \qquad \into |\D \vp_n|^p=\into |u_n|^r\vp_n.
\end{equation*}
Then
\begin{equation}
\label{s0r}
\into |\D u_n|^p+\into |\D \vp_n|^p \leq C\into f_nu_n.
\end{equation}
Choosing $u_n^+=u_n \chi_{\{u_n\geq 0\}}$ as test function in $(ii)$ we obtain
\begin{equation}
\label{s1r}
\into |\D\vp_n|^{p-2}\D\vp_n\cdot\D u_n^+=\into |u_n|^ru_n^+=\into |u_n^+|^{r+1}.
\end{equation}
For the term on the left hand side of \eqref{s1r} we have, by Young's inequality and \eqref{s0r}, that
\begin{align}
\label{s2r}
&\into |\D\vp_n|^{p-2}\D\vp_n\cdot\D u_n^+\leq \frac{1}{p'}\into |\D\vp_n|^p+ \frac{1}{p}\into |\D u_n^+|^p \\
&\leq \frac{1}{p'}\into |\D\vp_n|^p+ \frac{1}{p}\into |\D u_n|^p \leq C\into f_n u_n.\nonumber
\end{align}
Putting together \eqref{s1r} and \eqref{s2r}, we obtain
\begin{equation*}
\into |u_n^+|^{r+1}\leq C\into f_n u_n.
\end{equation*}
In the same way, using $u_n^-=-u_n\chi_{\{u_n<0\}}$ as test function in $(ii)$, we have
$$\into |u_n^-|^{r+1}\leq C\into f_n u_n,$$
so that
\begin{equation}
\label{s3r}
\into |u_n|^{r+1}=\into |u_n^+|^{r+1} +\into |u_n^-|^{r+1} \leq C\into f_nu_n \leq C\into |f||u_n|.
\end{equation}
Then, applying H\"older inequality to the right hand side of \eqref{s3r} with exponents $(r+1)'$ and $r+1$, we deduce 
\begin{equation}
\label{s4r}
\|u_n\|_{\lru} \leq C\|f\|_{\lr1d}^{\frac{1}{r}}.
\end{equation}
This implies, by \eqref{s0r} and H\"older's inequality, that
\begin{equation}
\label{s5r}
\into |\D u_n|^p+\into |\D \vp_n|^p \leq C \|f\|_{\lr1d} \|u_n\|_{\lru} \leq C \|f\|_{\lr1d}^{\frac{r+1}{r}},
\end{equation}
and
\begin{equation}
\label{s6r}
\into \vp_n |u_n|^r\leq C \|f\|_{\lr1d}^{\frac{r+1}{r}}.
\end{equation}

As a consequence of \eqref{s4r}, \eqref{s5r} and \eqref{s6r}, we have the following lemma.
\begin{lemma}
\label{lemr}
Let $f$ in $\lr1d$, and let $A>0$ and $r>1$. Then the weak solution $(u_n,\vp_n)$ of \eqref{papp1} is such that
\begin{equation*}
\|u_n\|_{\lru}+\|u_n\|_{\sob}+\|\vp_n\|_{\sob}+\into \vp_n |u_n|^r\leq C(f),
\end{equation*}
where $C(f)$ is a positive constant depending only on $\dis \|f\|_{\lr1d}$.
\end{lemma}

The above lemma implies that there exist subsequences still indexed by $u_n$ and $\vp_n$ and functions $u$ and $\vp$ belonging to $\sob$ such that
\begin{align}
\label{conr}
&u_n \rightarrow u \text{ weakly in } \sob, \text{ and a.e. in } \Omega, \nonumber \\
&u_n \rightarrow u \text{ weakly in } \lru, \text{ and strongly in } L^{q}(\Omega)\text{ } \forall q<\max\{r+1,p^*\}, \\
&\vp_n \rightarrow \vp \text{ weakly in } \sob, \text{ and a.e. in } \Omega. \nonumber
\end{align}
By applying these convergence results, we can prove the following existence theorem. 
\begin{theorem}
\label{teo2}
Let $A>0$, and let $r>1$ and $f$ in $L^m(\Omega)$, with $m\geq (r+1)'$. Then there exists a weak solution $(u,\vp)$ of system \eqref{pbp1}, with $u$ and $\vp$ in $\sob$.
\end{theorem}
The proof is a consequence of the proof of Theorem \ref{teo1} in the case $\theta=0$. We deduce, by Theorem \ref{teo2}, the regularizing effect for the solutions of \eqref{pbp1}. We assume 
\begin{equation}
\label{ip}
(r+1)'<(p^*)' \Leftrightarrow r>\displaystyle \frac{N(p-1)+p}{N-p} \quad\text{and} \quad f\in L^m(\Omega), \text{ with } m\geq (r+1)'.
\end{equation}
\begin{remark}
\label{main}
Under these assumptions we note that, if $m\geq (p^*)'$, thanks to the results in \cite{bg}, we have that $u$ belongs to $\sob\cap L^{t}(\Omega)$, with $t:=\displaystyle \frac{Nm(p-1)}{N-pm}$. Then, if $\displaystyle \frac{t}{r}<(p^*)'$, that is $\displaystyle m<m_1:=\frac{Npr}{N(p-1)^2+p(p-1)+p^2r}$, $\vp$ belongs to $\sob$ even if the datum of the second equation of \eqref{pbp1} does not belongs to the dual space. We verify that $m_1>(p^*)'$. Since 
$$m_1=\frac{pNr}{N(p-1)^2+p(p-1)+p^2r}> (p^*)'=\frac{Np}{N(p-1)+p} \Leftrightarrow r>p^*-1,$$
it follows thanks to \eqref{ip}. Moreover we have that, if $m<(p^*)'$ (i.e. the datum $f$ does not belong to $\dualsob$), then $u$ belongs to $\sob$. Hence we have a regularizing effect due to the system: the functions $u$ and $\vp$ belong to $\sob$ because of the coupling between the equations. This fact does not follow on being solutions of the single equations. 
\end{remark}

We now prove summability results for $u$. 
\begin{prop}
\label{p2}
Under the assumptions \eqref{ip}, the weak solution $u$ of \eqref{pbp1}, given by Theorem \ref{teo2}, belongs to $L^s(\Omega)$, with $\displaystyle s=\frac{m(pr+p-1)}{m(p-1)+1}$.
\end{prop}

\begin{proof}
We recall that $u$ is obtained from \eqref{conr} and that $(u_n,\varphi_n)$ is a weak solution of the system \eqref{papp1}. Choosing $(u_n^+)^\gamma$ as test function in $(ii)$ of \eqref{papp1}, with $\gamma \geq 1$, we have
\begin{equation}
\label{sum1}
\gamma\into |\D\vp_n|^{p-2}\D\vp_n\cdot\D u_n^+(u_n^+)^{\gamma-1}=\into (u_n^+)^{r+\gamma}.
\end{equation}
Applying Young's inequality to the left hand side of \eqref{sum1} we obtain, by Lemma \ref{lemr}, that
\begin{align}
\label{sum2}
\gamma\into |\D\vp_n|^{p-2}\D\vp_n\cdot\D u_n^+(u_n^+)^{\gamma-1}&\leq C\into |\D\vp_n|^p+C\into |\D u_n|^p (u_n^+)^{p(\gamma-1)} \\
&=C(f)+C\into |\D u_n|^p (u_n^+)^{p\gamma-p}.\nonumber
\end{align}
Now using $(u_n^+)^{p\gamma-p+1}$ as test function in $(i)$ of \eqref{papp1} we have, by H\"older's inequality, that
\begin{align}
\label{sum3}
\into |\D u_n^+|^p (u_n^+)^{p\gamma-p}&\leq C\into |\D u_n^+|^p (u_n^+)^{p\gamma-p}+C\into\vp_n (u_n^+)^{r+p\gamma-p}\\
&\leq C\into f_n(u_n^+)^{p\gamma-p+1}\leq C \|f\|_{L^m(\Omega)}\left(\into (u_n^+)^{m'(p\gamma-p+1)}\right)^{\frac{1}{m'}}.\nonumber
\end{align}
As a consequence of \eqref{sum1}, \eqref{sum2} and \eqref{sum3} we obtain
\begin{equation}
\label{sum4}
\into (u_n^+)^{r+\gamma}\leq C(f)+C\|f\|_{L^m(\Omega)}\left(\into (u_n^+)^{m'(p\gamma-p+1)}\right)^{\frac{1}{m'}}.
\end{equation}
Imposing $r+\gamma=m'(p\gamma-p+1)$ we have 
$$\gamma=\frac{r(m-1)+m(p-1)}{m(p-1)+1} \quad \text{and} \quad s:=r+\gamma=\frac{m(pr+p-1)}{m(p-1)+1}.$$
We verify that $\gamma\geq 1$:
$$\gamma=\frac{r(m-1)+m(p-1)}{m(p-1)+1}\geq 1 \Leftrightarrow m\geq \frac{r+1}{r}=(r+1)',$$
which it is true by \eqref{ip}. Then, by \eqref{sum4}, we deduce
\begin{equation*}
\displaystyle \|u_n^+\|_{L^s(\Omega)}\leq C(f),
\end{equation*}
where $C(f)$ is a positive constant depending only on $\|f\|_{L^{m}(\Omega)}$. In the same way we obtain, using $u_n^-$ as test function, that 
\begin{equation*}
\displaystyle \|u_n^-\|_{L^s(\Omega)}\leq C(f).
\end{equation*}
Then we have 
\begin{equation*}
\displaystyle \|u_n\|_{L^s(\Omega)}=\|u_n^+\|_{L^s(\Omega)}+\|u_n^-\|_{L^s(\Omega)}\leq C(f),
\end{equation*}
and $u_n$ converges to $u$ weakly in $L^s(\Omega)$, so that $u\in L^s(\Omega)$.
\end{proof}

\begin{remark}
Comparing this summability result on $u$ with the result contained in \eqref{conr} we observe that
$$s=\frac{m(pr+p-1)}{m(p-1)+1}\geq r+1 \Leftrightarrow m\geq \frac{r+1}{r}=(r+1)',$$
then, if \eqref{ip} holds, $L^{s}(\Omega)\subset L^{r+1}(\Omega)$. Moreover, if $m\geq (p^*)'$, it follows from \cite{bg} that $u$ belongs to $L^{t}(\Omega)$, with $\dis t=\frac{Nm(p-1)}{N-pm}$. We have that
$$s\geq t \Leftrightarrow m\leq m_1.$$
Summarizing we obtain that the best summability results for $u$ are
\begin{equation}
\label{sum5}
u\in L^s(\Omega), \quad \text{if } (r+1)'\leq m< m_1,
\end{equation}
and
\begin{equation*}
u\in L^{t}(\Omega), \quad \text{if } m\geq m_1.
\end{equation*}
Then we note, by \eqref{sum5}, that we have also a regularizing effect for the summability of the solution $u$.
\end{remark}

\section{\textbf{Existence and regularizing effect in the dual case}}
\label{Sec4}

We prove now the existence theorem for a weak solution of \eqref{pbp} for $\theta\geq 0$ and $f$ belonging to $L^{(p^*)'}(\Omega)$. Let $\{f_n\}$ be a sequence that satisfies 
\begin{equation*}
\label{appd}
f_n\in\linf \text{, } |f_n|\leq |f| \text{ }\forall n\in\NN \text{ and } f_n\rightarrow f \text{ strongly in } L^{(p^*)'}(\Omega).
\end{equation*}
Then, by Proposition \ref{p1}, there exists a solution $(u_n,\vp_n)$ in $\sob \times \sob$ of the system
\begin{equation}
\label{papp}
\begin{cases}
-\div(|\D u_n|^{p-2}\D u_n)+A\vp_n^{\theta+1}|u_n|^{r-2}u_n=f_n, \quad &(I),\\
-\div(|\D\varphi_n|^{p-2}\D\varphi_n)=|u_n|^r\vp_n^\theta, \quad &(II),
\end{cases}
\end{equation}
with $\vp_n\geq 0$, $u_n$ and $\vp_n$ in $\linf$. Choosing $u_n$ as test function in $(I)$ and $\vp_n$ in $(II)$ we have
\begin{equation}
\label{s0d}
\into |\D u_n|^p+A\into \vp_n^{\theta+1}|u_n|^r=\into f_nu_n, \qquad \into |\D \vp_n|^p=\into |u_n|^r\vp_n^{\theta+1}.
\end{equation}
Then
\begin{equation}
\label{s1d}
\into |\D u_n|^p+\into |\D \vp_n|^p \leq C\into f_nu_n.
\end{equation}
We obtain, by \eqref{s1d} and applying H\"older's inequality and the Sobolev embedding theorem, that
\begin{align*}
&\into |\D u_n|^p\leq \into |\D u_n|^p+\into |\D \vp_n|^p \leq C\into f_nu_n \\
&\leq C\|f\|_{L^{(p^*)'}(\Omega)}\|u_n\|_{L^{p^*}(\Omega)}\leq C\|f\|_{L^{(p^*)'}(\Omega)}\|u_n\|_{\sob}, \nonumber
\end{align*}
so that
\begin{equation}
\label{s2d}
\|u_n\|_{\sob}\leq C\|f\|_{L^{(p^*)'}(\Omega)}^{\frac{1}{p-1}} \quad \text{ and }\quad \|\vp_n\|_{\sob}\leq C\|f\|_{L^{(p^*)'}(\Omega)}^{\frac{1}{p-1}}.
\end{equation}
Moreover, by \eqref{s0d}, we deduce
\begin{equation}
\label{s3d}
\into \vp_n^{\theta+1}|u_n|^r\leq C\|f\|_{L^{(p^*)'}(\Omega)}^{\frac{p}{p-1}}.
\end{equation}
Choosing $u_n^+$ as test function in $(II)$ we obtain
\begin{equation*}
\label{s4d}
\into |\D\vp_n|^{p-2}\D\vp_n\cdot\D u_n^+=\into |u_n|^ru_n^+\vp_n^\theta=\into |u_n^+|^{r+1}\vp_n^\theta.
\end{equation*}
Using Young's inequality and \eqref{s2d}, we find
\begin{align*}
\label{s5d}
\into |u_n^+|^{r+1}\vp_n^\theta &=\into |\D\vp_n|^{p-2}\D\vp_n\cdot\D u_n^+\leq \frac{1}{p'}\into |\D\vp_n|^p+ \frac{1}{p}\into |\D u_n^+|^p \\
&\leq \frac{1}{p'}\into |\D\vp_n|^p+ \frac{1}{p}\into |\D u_n|^p \leq C\|f\|_{L^{(p^*)'}(\Omega)}^{\frac{p}{p-1}}.\nonumber
\end{align*}
In the same way, choosing $u_n^-$ as test function in $(II)$, we deduce
$$\into |u_n^-|^{r+1}\vp_n^\theta\leq C\|f\|_{L^{(p^*)'}(\Omega)}^{\frac{p}{p-1}},$$
so that
\begin{equation}
\label{s6d}
\into |u_n|^{r+1}\vp_n^\theta=\into |u_n^+|^{r+1}\vp_n^\theta +\into |u_n^-|^{r+1}\vp_n^\theta \leq C\|f\|_{L^{(p^*)'}(\Omega)}^{\frac{p}{p-1}}.
\end{equation}

As a consequence of \eqref{s2d}, \eqref{s3d} and \eqref{s6d}, we have the following lemma.
\begin{lemma}
\label{lemd}
Let $f$ in $L^{(p^*)'}(\Omega)$, and let $A>0$, $r>1$ and $0\leq\theta<p-1$. Then the weak solution $(u_n,\vp_n)$ of \eqref{papp}, given by Proposition \ref{p1}, is such that
\begin{equation*}
\|u_n\|_{\sob}+\|\vp_n\|_{\sob}+\into \vp_n^{\theta+1}|u_n|^r+\into |u_n|^{r+1}\vp_n^\theta\leq C(f)
\end{equation*}
where $C(f)$ is a positive constant depending only on $\dis \|f\|_{L^{(p^*)'}(\Omega)}$.
\end{lemma}

Once again, by Lemma \ref{lemd}, there exist subsequences still indexed by $u_n$ and $\vp_n$ and functions $u$ and $\vp$ in $\sob$ such that
\begin{align}
\label{cond}
\begin{array}{l}
u_n \rightarrow u \text{ weakly in } \sob, \text{ strongly in } L^q(\Omega),\text{ with }q<p^*,\text{ and a.e. in } \Omega, \\
\vp_n \rightarrow \vp \text{ weakly in } \sob,\text{ strongly in } L^q(\Omega),\text{ with }q<p^*, \text{ and a.e. in } \Omega. 
\end{array}
\end{align}

\begin{theorem}
\label{teo1}
Let $A>0$, and let $r>1$, $0\leq\theta<p-1$ and $f$ in $L^m(\Omega)$, with $m\geq (p^*)'$. Then there exists a weak solution $(u,\vp)$ in $\sob\times\sob$ of system \eqref{pbp}.
\end{theorem}

\begin{proof}
Let $u$ and $\vp$ be the functions defined in \eqref{cond}. We want to pass to the limit in $(II)$ of \eqref{papp}. We recall that $\vp_n$ satisfies
\begin{equation}
\label{td0}
\into |\D\vp_n|^{p-2}\D\vp_n \cdot\D\psi=\into |u_n|^r\vp_n^\theta\psi, \quad \forall \psi\in\sob.
\end{equation}
We want to prove that $|u_n|^r\vp_n^\theta$ strongly converges to $|u|^r\vp^\theta$ in $L^1(\Omega)$. Fix $\sigma>0$ and let $E\subset\Omega$. By Lemma \ref{lemd} there exists $\overline{k}\in\NN$ such that  
\begin{align*}
\label{td1}
\int_E |u_n|^r\vp_n^\theta &=\int_{E\cap\{|u_n|\leq\overline{k}\}} |u_n|^r\vp_n^\theta+\int_{E\cap\{|u_n|>\overline{k}\}} |u_n|^r\vp_n^\theta \leq \overline{k}^r\int_E \vp_n^\theta+\frac{1}{\overline{k}}\int_{\{|u_n|>\overline{k}\}}|u_n|^{r+1}\vp_n^\theta \\
&\leq \overline{k}^r\int_E \vp_n^\theta +\frac{C(f)}{\overline{k}}\leq \overline{k}^r\int_E \vp_n^\theta +\frac{\sigma}{2}.\nonumber
\end{align*}
Since, by \eqref{cond}, $\vp_n^\theta$ strongly converges to $\vp^\theta$ in $L^1(\Omega)$, applying Vitali's theorem, there exists $\delta>0$ such that $|E|<\delta$ and 
$$\int_E |u_n|^r\vp_n^\theta\leq \overline{k}^r\int_E \vp_n^\theta +\frac{\sigma}{2}\leq \sigma.$$
Then, once again using Vitali's theorem, we have 
\begin{equation}
\label{td2}
|u_n|^r\vp_n^\theta \rightarrow |u|^r\vp^\theta \text{ strongly in } L^1(\Omega).
\end{equation}
Hence, by Theorem 2.1 in \cite{bm}, we obtain that $\D\vp_n$ converges $\D\vp$ almost everywhere in $\Omega$. Moreover
$$\||\D\vp_n|^{p-2}\D\vp_n\|_{\lp'n}\leq \|\vp_n\|_{\sob}^{p-1}\leq C(f),$$
so that
\begin{equation}
\label{td3}
|\D\vp_n|^{p-2}\D\vp_n\rightarrow |\D\vp|^{p-2}\D\vp \text{ weakly in }\lp'n.
\end{equation}
Fix $\psi$ in $\sob\cap\linf$, we have, by \eqref{td3}, that
$$\lim_{n\to\infty}\into|\D\vp_n|^{p-2}\D\vp_n\cdot\D\psi=\into |\D\vp|^{p-2}\D\vp\cdot\D\psi.$$
On the other hand, by \eqref{td2} and Vitali's theorem, we find
$$\lim_{n\to\infty}\into |u_n|^r\vp_n^\theta\psi=\into |u|^r\vp^\theta\psi.$$
By passing to the limit in \eqref{td0}, we obtain that
\begin{equation}
\label{td4}
\into |\D\vp|^{p-2}\D\vp\cdot\D\psi=\into |u|^r\vp^\theta\psi, \quad \forall \psi\in\sob\cap\linf.
\end{equation}
Let $\eta$ belong to $\sob$. Choosing $\psi=T_k(\eta)$ as test function in \eqref{td4}, we obtain
\begin{equation}
\label{td5}
\into |\D\vp|^{p-2}\D\vp\cdot\D T_k(\eta)=\into |u|^r\vp^\theta T_k(\eta). 
\end{equation}
We have that $\dis |\D\vp|^{p-2}\D\vp\cdot\D T_k(\eta)$ converges to $\dis |\D\vp|^{p-2}\D\vp\cdot\D\eta$ almost everywhere in $\Omega$ and that
$$\left ||\D\vp|^{p-2}\D\vp\cdot\D T_k(\eta)\right |\leq \left ||\D\vp|^{p-2}\D\vp\cdot\D\eta\right |,$$
with $\dis \left ||\D\vp|^{p-2}\D\vp\cdot\D\eta \right |$ in $\luno$.
Then, by Lebesgue's theorem, we deduce
\begin{equation}
\label{td6}
\lim_{k\to\infty}\into |\D\vp|^{p-2}\D\vp\cdot\D T_k(\eta)=\into |\D\vp|^{p-2}\D\vp\cdot\D\eta.
\end{equation}
Now we want to let $k$ to infinity on the right hand side of \eqref{td5}. We recall that
$$|u|^r\vp^\theta T_k(\eta)=|u|^r\vp^\theta T_k(\eta^+)-|u|^r\vp^\theta T_k(\eta^-),$$
where $|u|^r\vp^\theta T_k(\eta^+)$ and $|u|^r\vp^\theta T_k(\eta^-)$ are nonnegative functions increasing in $k$. We have that $|u|^r\vp^\theta T_k(\eta^+)$ converges to $|u|^r\vp^\theta\eta^+$ and $|u|^r\vp^\theta T_k(\eta^-)$ converges to $|u|^r\vp^\theta\eta^-$ almost everywhere in $\Omega$. It follows from Beppo Levi's theorem that 
\begin{equation*}
\lim_{k\to\infty}\into |u|^r\vp^\theta T_k(\eta^+)=\into |u|^r\vp^\theta\eta^+ \text{ and } \lim_{k\to\infty}\into |u|^r\vp^\theta T_k(\eta^-)=\into |u|^r\vp^\theta\eta^-,
\end{equation*}
so that
\begin{align}
\label{td7}
&\lim_{k\to\infty}\into |u|^r\vp^\theta T_k(\eta)=\lim_{k\to\infty}\into |u|^r\vp^\theta T_k(\eta^+)-\lim_{k\to\infty}\into |u|^r\vp^\theta T_k(\eta^-) \\
&=\into |u|^r\vp^\theta \eta^+ -\into |u|^r\vp^\theta\eta^-=\into |u|^r\vp^\theta\eta. \nonumber
\end{align}
Letting $k$ to infinity in \eqref{td5}, by \eqref{td6} and \eqref{td7}, we obtain
\begin{equation*}
\label{td8}
\into |\D\vp|^{p-2}\D\vp\cdot\D\eta=\into |u|^r\vp^\theta\eta, \quad \forall\eta\in\sob.
\end{equation*}
Then $\vp$ in $\sob$ is a weak solution of the second equation of \eqref{pbp}. \\
Now we want to pass to the limit in $(I)$ of \eqref{papp}. We have that $u_n$ satisfies
\begin{equation}
\label{td9}
\into |\D u_n|^{p-2}\D u_n\cdot\D\psi+A\into \vp_n^{\theta+1} |u_n|^{r-2}u_n\psi=\into f_n\psi, \quad \forall\psi\in\sob.
\end{equation}
Fix $\eps>0$. Choosing $\dis \psi=\frac{T_\eps(G_k(u_n))}{\eps}$ in \eqref{td9}, we obtain
\begin{equation*}
\frac{1}{\eps} \int_{\{k\leq |u_n|\leq k+\eps\}} |\D u_n|^p+A\int_{\{|u_n|\geq k\}} \vp_n^{\theta+1} |u_n|^{r-2}u_n \frac{T_\eps(G_k(u_n))}{\eps}=\int_{\{|u_n|\geq k\}} f_n \frac{T_\eps(G_k(u_n))}{\eps}.
\end{equation*}
Dropping the first nonnegative term, we have
\begin{align*}
&A\int_{\{|u_n|\geq k+\eps\}}\vp_n^{\theta+1} |u_n|^{r-1}\leq A\int_{\{|u_n|\geq k\}} \vp_n^{\theta+1} |u_n|^{r-2}u_n \frac{T_\eps(G_k(u_n))}{\eps} \\
&\leq\int_{\{|u_n|\geq k\}} |f_n|\left |\frac{T_\eps(G_k(u_n))}{\eps}\right |\leq\int_{\{|u_n|\geq k\}}|f_n|,
\end{align*}
so that
\begin{equation*}
A\int_{\{|u_n|\geq k+\eps\}}\vp_n^{\theta+1} |u_n|^{r-1}\leq\int_{\{|u_n|\geq k\}}|f|.
\end{equation*}
Letting $\eps$ tend to zero, by Beppo Levi's theorem, we obtain
\begin{equation}
\label{td10}
\int_{\{|u_n|\geq k\}}\vp_n^{\theta+1} |u_n|^{r-1}\leq\frac{1}{A}\int_{\{|u_n|\geq k\}}|f|.
\end{equation}
Once again fix $\sigma>0$ and let $E\subset\Omega$. By \eqref{td10}, we have
\begin{align*}
&\int_E \vp_n^{\theta+1}|u_n|^{r-1}=\int_{E\cap\{|u_n|\leq k\}}\vp_n^{\theta+1}|u_n|^{r-1}+\int_{E\cap\{|u_n|>k\}}\vp_n^{\theta+1}|u_n|^{r-1} \\
&\leq k^{r-1}\int_{E} \vp_n^{\theta+1} +\frac{1}{A}\int_{\{|u_n|\geq k\}}|f|. 
\end{align*}
As a consequence of \eqref{cond} and applying Vitali's theorem, there exist $\tilde{k}$ and $\delta>0$, with $|E|<\delta$, such that
\begin{equation*}
\frac{1}{A}\int_{\{|u_n|\geq \tilde{k}\}}|f|\leq \frac{\sigma}{2} \quad \text{ and }\quad \tilde{k}^{r-1}\int_E \vp_n^{\theta+1}\leq\frac{\sigma}{2},
\end{equation*}
uniformly in $n$. Then we deduce
\begin{equation}
\label{td11}
\int_E \vp_n^{\theta+1}|u_n|^{r-1}\leq \sigma,
\end{equation}
uniformly in $n$. We recall that, by \eqref{cond}, $\vp_n^{\theta+1} |u_n|^{r-1}$ converges to $\vp^{\theta+1} |u|^{r-1}$ almost everywhere in $\Omega$. Thanks to \eqref{td11}, applying Vitali's theorem, we obtain that
\begin{equation}
\label{td12}
\vp_n^{\theta+1} |u_n|^{r-1} \rightarrow \vp^{\theta+1} |u|^{r-1} \text{ strongly in } \luno.
\end{equation}
We have that 
$$-\div(|\D u_n|^{p-2}\D u_n)=-A\vp_n^{\theta+1} |u_n|^{r-2}u_n+f_n=:g_n,$$
and, by the assumptions on $f$ and \eqref{td12}, that $\dis \|g_n\|_{\luno}\leq C$. Applying Theorem 2.1 in \cite{bm}, we obtain that $\D u_n$ converges to $\D u$ almost everywhere in $\Omega$. Moreover 
$$\||\D u_n|^{p-2}\D u_n\|_{\lp'n}\leq \|u\|_{\sob}^{p-1}\leq C(f),$$
then
\begin{equation}
\label{td13}
|\D u_n|^{p-2}\D u_n\rightarrow |\D u|^{p-2}\D u \text{ weakly in } \lp'n.
\end{equation}
By passing to the limit as $n$ tends to infinity in \eqref{td9}, by \eqref{td12} and \eqref{td13}, and applying Lebesgue's theorem, we deduce that
\begin{equation*}
\into |\D u|^{p-2}\D u\cdot \D\psi+A\into \vp^{\theta+1} |u|^{r-2}u\psi=\into f\psi, \quad \forall\psi\in\sob\cap\linf.
\end{equation*}
Proceeding as when we passed to the limit in $(II)$, we have
$$\into |\D u|^{p-2}\D u\cdot \D v+A\into \vp^{\theta+1} |u|^{r-2}u v=\into fv, \quad \forall v\in\sob.$$
Then $u$ in $\sob$ is a weak solution of the first equation of \eqref{pbp} and $(u,\vp)$ is a weak solution of \eqref{pbp}.
\end{proof}

\begin{remark}
\label{finalrm2}
We want to stress the fact that, in order to prove this theorem, we only used the results \eqref{cond} obtained as consequence of the estimates in Lemma \ref{lemd}. Since the results \eqref{conr} are analogous, proceeding in the same way we can prove, as said before, Theorem \ref{teo2}.
\end{remark}

\begin{remark}
We observe that, thanks to the results in \cite{bo1}, the second equation of \eqref{pbp} admits a weak solution in $\sob$ if $\dis |u|^r\in L^s(\Omega)$, with $\dis s\geq \left(\frac{p^*}{\theta+1}\right)'$. We recall that $u$ belongs to $L^t(\Omega)$, with $\displaystyle t=\frac{Nm(p-1)}{N-pm}$. Then, if $\displaystyle \frac{t}{r}<\left(\frac{p^*}{\theta+1}\right)'$, we deduce once again a regularizing effect on $\vp$ due to the coupling in the system. We have that 
$$\frac{t}{r}<\left(\frac{p^*}{\theta+1}\right)' \Leftrightarrow m<m_2:=\frac{Npr}{N(p-1)^2+p(p-1)+p^2r-\theta(p-1)(N-p)}.$$
For this to be possible we must have that $r>p^*-1-\theta$. We stress the fact that for $\theta=0$ we recover the regularizing effect on $\vp$ observed in Remark \ref{main}.
\end{remark}

In this case ($\theta>0$) we are not able to prove a regularizing effect on the existence of a finite energy solution for the first equation of \eqref{pbp}. We feel that this is an obstacle only due to the method used, and that the following conjecture should be true.
\begin{conj}
\label{finalrm0}
Let $A>0$, and let $r>1$ and $0\leq\theta<p-1$. Then there exists $1<\underline{m}<(p^*)'$ such that if $f$ belongs to $L^m(\Omega)$, with $m\geq\underline{m}$, then there exists a weak solution $(u,\vp)$ in $\sob\times\sob$ of system \eqref{pbp}.
\end{conj}
For instance if we assume that $|u_n|\leq c\,\vp_n$ in $\Omega$, for some $c>0$, we are able to prove that this conjecture is true with $\underline{m}=(r+1+\theta)'$ and $r>p^*-1-\theta$. Indeed, if we consider the approximate problem \eqref{papp}, choosing $u_n^+$ as test function in $(II)$, we obtain
\begin{equation}
\label{finalrm}
\into |\D\vp_n|^{p-2}\D\vp_n\cdot\D u_n^+\,=\,\into |u_n|^r\vp_n^\theta u_n^+\,=\,\into |u_n^+|^{r+1}\vp_n^\theta\,\geq\frac{1}{c^\theta}\into |u_n^+|^{r+1+\theta}\,.
\end{equation}
So, by Young's inequality, using \eqref{s1d} and applying H\"older's inequality, we deduce from \eqref{finalrm} that
\begin{align*}
&\frac{1}{c^\theta}\into |u_n^+|^{r+1+\theta}\,\leq\,\into |\D\vp_n|^{p-2}\D\vp_n\cdot\D u_n^+\,\leq\, \frac{1}{p'}\into |\D\vp_n|^p+ \frac{1}{p}\into |\D u_n^+|^p \\
&\leq\, \frac{1}{p'}\into |\D\vp_n|^p+ \frac{1}{p}\into |\D u_n|^p\, \leq\, C\into |f| |u_n|\,\leq\,C\|f\|_{L^{(r+1+\theta)'}(\Omega)}\|u_n\|_{L^{r+1+\theta}(\Omega)}\,.
\end{align*}
Thus we have, once again, that
$$
\|u_n\|_{\sob}+\|\vp_n\|_{\sob}+\into \vp_n^{\theta+1}|u_n|^r+\into |u_n|^{r+1}\vp_n^\theta\leq C(f)\,,
$$
where $C(f)$ is a positive constant depending only on $\dis \|f\|_{L^{(r+1+\theta)'}(\Omega)}$. \\
Thanks to these estimates it follows from Remark \ref{finalrm2} that we can pass to the limit in \eqref{papp}. Hence we have proved our conjecture with $\underline{m}=(r+1+\theta)'$. \\
We note that for $\theta=0$ we obtain $\underline{m}=(r+1)'$, that is, exactly, the result stated in Theorem \ref{teo2}.

\end{document}